\documentclass[oneside,final,11pt]{amsart}

\usepackage{dsfont}

\newcommand{\R}{\mathds R}
\newcommand{\dd}{\mathrm d}
\newcommand{\EP}{$c_0$EP}
\newcommand{\DA}{\mathrm{DA}}

\newcommand{\upstar}[1]{{#1}^{\raise1pt\hbox{$\scriptscriptstyle*$}}}
\newcommand{\chilow}[1]{\chi_{\lower2pt\hbox{$\scriptstyle#1$}}}

\DeclareMathOperator{\Ker}{Ker}
\DeclareMathOperator{\BV}{BV}
\DeclareMathOperator{\NBV}{NBV}
\DeclareMathOperator{\diam}{diam}
\DeclareMathOperator{\supp}{supp}

\title{Compact lines and the Sobczyk property}

\author{Claudia Correa}
\thanks{The first author is sponsored by FAPESP (Process no.\ 2012/25171-0).}
\address{Departamento de Matem\'atica,\hfill\break\indent Universidade de S\~ao Paulo, Brazil}
\email{claudiac.mat@gmail.com}
\author{Daniel V. Tausk}
\address{Departamento de Matem\'atica,\hfill\break\indent Universidade de S\~ao Paulo, Brazil}
\email{tausk@ime.usp.br} \urladdr{http://www.ime.usp.br/\~{}tausk}

\date{February 9th, 2014}

\begin{document}

\theoremstyle{plain}\newtheorem{teo}{Theorem}[section]
\theoremstyle{plain}\newtheorem{prop}[teo]{Proposition}
\theoremstyle{plain}\newtheorem{lem}[teo]{Lemma}
\theoremstyle{plain}\newtheorem{cor}[teo]{Corollary}
\theoremstyle{definition}\newtheorem{defin}[teo]{Definition}
\theoremstyle{remark}\newtheorem{rem}[teo]{Remark}
\theoremstyle{plain} \newtheorem{assum}[teo]{Assumption}
\theoremstyle{definition}\newtheorem{example}[teo]{Example}

\begin{abstract}
We show that Sobczyk's Theorem holds for a new class of Banach spaces, namely spaces of continuous functions on linearly ordered compacta.
\end{abstract}

\maketitle

\begin{section}{Introduction}

The main result of this paper (Corollary~\ref{thm:obaobaSobczyk}) states that, for every compact line $K$, the Banach space $C(K)$ of continuous real-valued functions on $K$ (endowed with the supremum norm) has the {\em Sobczyk property}, i.e., every isomorphic copy of $c_0$ in $C(K)$ is complemented. By a {\em compact line\/} we mean a linearly ordered set which is compact in the order topology (see \cite[1.7.4]{Engelking} for the definition
and basic facts about the order topology
and \cite[3.12.3]{Engelking} for the characterization of linear orders yielding a compact topology). Topological properties of compact lines and structural properties of their spaces of continuous functions have recently been studied in a series of articles \cite{CT, KK, KK2, Kubis}.

In his celebrated theorem \cite{Sobczyk}, Sobczyk has proven that every separable Banach space has the property now named after him. Such result follows from the fact that $c_0$ is {\em separably injective}, i.e., given a separable
Banach space $X$ and a closed subspace $Y$ of $X$, every bounded operator $T:Y\to c_0$ admits a bounded extension to $X$. When the latter
condition holds for a closed subspace $Y$ of a Banach space $X$, we say that $Y$ has the {\em $c_0$-extension property\/} (briefly: \EP)
in $X$. Moreover, a Banach space $X$ is said to have the (resp., {\em separable}) {\em $c_0$-extension property\/} if every (resp., separable)
closed subspace of $X$ has the \EP\ in $X$. The \EP\ for Banach spaces has been studied by the authors in \cite{CT}.

The main result of the present paper follows from our Theorem~\ref{thm:main}
which states that, if $K$ is a compact line, then Banach subspaces of $C(K)$ with separable dual have the \EP\ in $C(K)$.
The latter generalizes \cite[Theorem~3.1]{CT}, which implies that, for a compact line $K$, Banach subalgebras of $C(K)$ with separable dual
have the \EP\ in $C(K)$ (though --- see Remark~\ref{thm:rem8} --- the extension constant obtained in \cite[Theorem~3.1]{CT} is smaller).
Note that, although separable Banach subspaces of $C(K)$ always span separable Banach subalgebras, even a one-dimensional
subspace of $C(K)$ can span a Banach subalgebra with nonseparable dual. (By the {\em Banach subalgebra spanned\/} by a subset of $C(K)$
we mean the smallest Banach subalgebra of $C(K)$ containing that set.) Thus Theorem~\ref{thm:main} is much stronger than \cite[Theorem~3.1]{CT}:
the latter does not imply that $C(K)$ has the Sobczyk property.

Let us briefly review some known results stating that certain classes of Banach spaces have the Sobczyk property. We remark that their proofs can usually be adapted to establish the \EP\ or the separable \EP. For instance, a simple adaptation of Veech's proof (\cite{Veech}) of Sobczyk's Theorem shows that weakly compactly generated (briefly: WCG) Banach spaces have the \EP\ (see also \cite[Proposition~2.2]{CT} for details).
Molt\'o's argument (\cite[Theorem~3]{Molto}) shows that, assuming a certain topological condition for the dual ball $(B_{X^*},\text{w*})$, one obtains that the Banach space $X$ has the separable \EP. (Molt\'o's topological condition is satisfied, for instance, by all Corson compacta.) A Banach space $X$
is said to satisfy the {\em separable complementation property\/} (briefly: SCP) if every separable subspace of $X$ is contained in a separable complemented Banach subspace of $X$. Sobczyk's Theorem implies that all Banach spaces with the SCP have the separable \EP. In \cite[Lemma, p.~494]{Valdivia} it is shown that if $K$ is a Valdivia compact space then $C(K)$ has the SCP. Since the separable \EP\ is hereditary to closed subspaces, it follows also that $C(K)$ has the separable \EP\ when $K$ is a continuous image of a Valdivia compactum.

Recall that a compact Hausdorff space $K$ is said to be {\em $\aleph_0$-monolithic\/} if every separable subspace of $K$
is metrizable. It is easy to show (see \cite[Corollary~2.7]{CT}) that if $K$ is $\aleph_0$-monolithic then $C(K)$ has the separable \EP. It turns out that
if $K$ is a compact line then $C(K)$ has the separable \EP\ only in the trivial case when $K$ is $\aleph_0$-monolithic
(Theorem~\ref{thm:monolithic}), though $C(K)$ always has the Sobczyk property.

The {\em double arrow space\/} $\DA=[0,1]\times\{0,1\}$ (endowed with the lexicographic order and order topology) is a separable nonmetrizable compact line and thus every compact line containing a homeomorphic copy of $\DA$ is not $\aleph_0$-monolithic. To the best of the authors knowledge, $C(\DA)$ is the first known example of a $C(K)$ space satisfying the Sobczyk property, yet not having the separable \EP. We note that Patterson \cite{Patterson} has shown that every {\em isometric\/} copy
of $c_0$ in $C(\DA)$ is complemented, but the problem of determining whether $C(\DA)$ has the Sobczyk property remained open. Patterson's argument uses only the fact that closed subsets of $\DA$ admit bounded extension operators and therefore
it works for arbitrary compact lines, though it is of no use to establish complementation of arbitrary {\em isomorphic\/} copies of $c_0$.
In fact, in his review \cite{Godefroy} of \cite{Patterson}, Godefroy mentions that no examples of uncomplemented isomorphic copies of $c_0$ in $C(\DA)$
are given in the paper. This observation led the authors to study the problem whose solution is presented here.

This paper is organized as follows. In Section~\ref{sec:c0EP} we give a characterization of the compact lines $K$
for which $C(K)$ has the (separable) \EP\ (Theorem~\ref{thm:monolithic}) and we prove the main theorem of the paper (Theorem~\ref{thm:main}). The heart of the proof of Theorem~\ref{thm:main} is Lemma~\ref{thm:hardoperators}, whose proof takes all of Section~\ref{sec:technical}.

\end{section}

\begin{section}{The $c_0$-extension property for compact lines}
\label{sec:c0EP}

We start by fixing the terminology and notation for the paper and by recalling a few elementary facts.

Given a compact Hausdorff space $K$, we identify as usual the dual space of $C(K)$ with the space $M(K)$ of finite countably-additive signed regular Borel measures on $K$, endowed with the total variation norm $\Vert\mu\Vert=\vert\mu\vert(K)$ (\cite[Theorem~6.19]{Rudin}). Given a point $p\in K$ and a subset $A$ of $K$, we denote by $\delta_p\in M(K)$ the probability measure with support $\{p\}$ and by $\chilow A$ the characteristic function of $A$.

It is easily proven using the Stone--Weierstrass Theorem (see, for instance, \cite[proof of Theorem~5.4]{Kozsmider}) that all Banach subalgebras (with unity) of $C(K)$ are images $q^*C(L)$ of composition operators $q^*:C(L)\ni f\mapsto f\circ q\in C(K)$,
where $L$ is a compact Hausdorff space and $q:K\to L$ is a continuous surjection. The operator $q^*$ is an isometric embedding and a Banach algebra
homomorphism.

Bounded operators $T:C(K)\to\ell_\infty$ are always identified with bounded sequences of measures $(\mu_n)_{n\ge1}$ in $M(K)$, where $\mu_n$
represents the $n$-th coordinate functional of $T$. In this case we will say that $T$ is associated with $(\mu_n)_{n\ge1}$.
Note that $T$ takes values in $c_0$ if and only if $(\mu_n)_{n\ge1}$ is weak*-null.

When $K$ is a compact line, we always denote by $0$ the minimum element of $K$. A point $t\in K$ is said to be {\em right-isolated\/} (resp.,
{\em left-isolated}) in $K$ if either $t$ is the maximum (resp., minimum) element of $K$ or $t$ admits a successor (resp., predecessor) in $K$.
The clopen subsets of $K$ are finite disjoint unions of clopen intervals of $K$, which are of the form $[0,b]$ or $\left]b',b\right]$, with
$b,b'\in K$ right-isolated.

Given a subset $Q$ of $[0,1]$, we denote by $\DA(Q)$ the set:
\[\DA(Q)=\big([0,1]\times\{0\}\big)\cup\big(Q\times\{1\}\big)\subset[0,1]\times\{0,1\}\]
endowed with the lexicographic order (and the order topology). Then $\DA(Q)$ is a separable compact line and the first projection $\pi_1:\DA(Q)\to[0,1]$ is a continuous increasing surjection. The proof of the next lemma uses a criterion for the extensibility of $c_0$-valued operators that will be proven later
in Section~\ref{sec:technical}.
\begin{lem}\label{thm:DAQ}
If $Q\subset[0,1]$ is uncountable, then $\pi_1^*C\big([0,1]\big)$ does not have the \EP\ in $C\big(\DA(Q)\big)$.
\end{lem}
\begin{proof}
Let $\big(\!\left[a_n,b_n\right[\big)_{n\ge1}$
be a sequence of intervals contained in $[0,1]$ such that $\lim_{n\to+\infty}(b_n-a_n)=0$ and such that every $t\in\left[0,1\right[$
belongs to infinitely many $\left[a_n,b_n\right[$. Setting $\mu_n=\delta_{a_n}-\delta_{b_n}$, then $(\mu_n)_{n\ge1}$ is weak*-null
in $M\big([0,1]\big)$ and thus we have an operator $T:C\big([0,1]\big)\to c_0$ associated with $(\mu_n)_{n\ge1}$. It follows
from the proof of Lemma~\ref{thm:extensioncriterion} (see Remark~\ref{thm:gambiarra}) that there is no bounded operator $T':C\big(\DA(Q)\big)\to c_0$
such that $T'\circ\pi_1^*=T$. This concludes the proof.
\end{proof}

\medskip

The following result characterizes compact lines $K$ for which $C(K)$ has the (separable) \EP.
\begin{teo}\label{thm:monolithic}
Let $K$ be a compact line. The following statements are equivalent:
\begin{itemize}
\item[(a)] $K$ is $\aleph_0$-monolithic;
\item[(b)] $C(K)$ has the \EP;
\item[(c)] $C(K)$ has the separable \EP.
\end{itemize}
\end{teo}
\begin{proof}
Assume (a) and let us prove (b). Let $T:X\to c_0$ be a bounded operator defined on a closed subspace $X$ of $C(K)$. Let $S:C(K)\to\ell_\infty$ be a bounded
extension of $T$. We will show that the quotient $C(K)/S^{-1}[c_0]$ is separable and it will follow from \cite[Proposition~2.2,~(a)]{CT} that $S\vert_{S^{-1}[c_0]}$ has a $c_0$-valued bounded extension to $C(K)$. The operator $S$ is associated with a bounded sequence $(\mu_n)_{n\ge1}$
in $M(K)$. It follows from \cite[Lemma~2.1]{KK} that $\supp\mu_n$ is separable for all $n$ and thus the closure $F$ of $\bigcup_{n=1}^\infty\supp\mu_n$
is also separable; by (a), $F$ is metrizable. The subspace $C(K|F)$ of $C(K)$ consisting of the functions that vanish on $F$ is clearly contained
in $\Ker(S)\subset S^{-1}[c_0]$. Since $C(K)/C(K|F)\equiv C(F)$ is separable, it follows that also $C(K)/S^{-1}[c_0]$ is separable.

Clearly, (b) implies (c). Now assume (c) and let us prove (a).
Assuming by contradiction that $K$ is not $\aleph_0$-monolithic,
it follows from \cite[Lemma~2.5]{KK} that $K$ contains a homeomorphic copy of $\DA(Q)$, for some uncountable subset $Q$ of $[0,1]$.
Since $K$ is a compact line, by \cite[Lemma~4.2]{Kubis}, every closed subset of $K$ admits a bounded extension operator and therefore
$C(K)$ contains an isomorphic copy of $C\big(\DA(Q)\big)$. The separable \EP\ is hereditary to closed subspaces and thus $C\big(\DA(Q)\big)$
would also have the separable \EP, contradicting Lemma~\ref{thm:DAQ}.
\end{proof}

\begin{rem}
As mentioned in the introduction, WCG Banach spaces have the \EP. Let us see that the converse does not hold.
Recall that a Banach space $C(K)$ is WCG if and only if $K$ is an Eberlein compact space (\cite[Theorem~14.9]{Fabian}). There are
examples of $\aleph_0$-monolithic compact lines that are not Eberlein compact spaces. For instance, an ordinal segment $[0,\alpha]$
is $\aleph_0$-monolithic; namely, it is easily proven by induction on $\alpha$ that the closure of a countable subset of $[0,\alpha]$ is countable. Moreover,
if $\alpha$ is uncountable, then $[0,\alpha]$ is not a Fr\'echet topological space (\cite[Definition~4.48]{Fabian2})
and hence it is not an Eberlein compact space (\cite[Theorem~4.50]{Fabian2}).
It follows from Theorem~\ref{thm:monolithic} that there are $C(K)$ spaces that have the \EP\ but are not WCG.
\end{rem}

Our main result follows trivially from the next theorem.
\begin{teo}\label{thm:main}
If $K$ is a compact line then every Banach subspace of $C(K)$ with separable dual has the \EP\ in $C(K)$.
\end{teo}

\begin{cor}\label{thm:obaobaSobczyk}
If $K$ is a compact line then the space $C(K)$ has the Sobczyk property.\qed
\end{cor}

The proof of Theorem~\ref{thm:main} requires two lemmas. The first one is a simple adaptation of \cite[Lemma~3.1]{KK}.
\begin{lem}\label{thm:lemaCL}
Let $K$ be a zero-dimensional compact line and $X$ be a separable Banach subspace of $C(K)$. Then there exists a metrizable zero-dimensional compact line
$L$ and a continuous increasing surjection $q:K\to L$ such that $X$ is contained in $q^*C(L)$.
\end{lem}
\begin{proof}
Since $K$ is zero-dimensional, the set:
\begin{equation}\label{eq:chi0t}
\big\{\chilow{[0,s]}:\text{$s$ right-isolated in $K$}\big\}
\end{equation}
spans a dense subspace of $C(K)$. By the separability of $X$, there exists a countable subset $E$ of \eqref{eq:chi0t} such that $X$ is contained in the
Banach space spanned by $E$. The Banach algebra spanned by $E$ is of the form $q^*C(L)$, with $L$ a compact metrizable space and $q:K\to L$ a continuous
surjection. Since the elements of $E$ are simple functions, the space $L$ is zero-dimensional. Finally, the fact that the elements of $E$ are monotone
functions implies that $q^{-1}(t)$ is a closed interval of $K$, for all $t\in L$. Hence there exists a unique order on $L$ such that $q$ is increasing
and the order-topology coincides with the topology of $L$.
\end{proof}

The hard work in proving Theorem~\ref{thm:main} lies within the proof of the next lemma. Such proof will be presented in Section~\ref{sec:technical}.
\begin{lem}\label{thm:hardoperators}
Let $K$ and $L$ be zero-dimensional compact lines, $q:K\to L$ be a continuous increasing surjection and $X$ be a Banach space with separable dual. Assume that $L$ is metrizable. Given bounded operators $T:C(L)\to c_0$, $R:X\to C(L)$ and $\varepsilon,\varepsilon'>0$, there exist bounded operators $T':C(K)\to c_0$ and $S:C(L)\to c_0$ such that:
\begin{gather*}
T=T'\circ q^*+S,\\
\Vert T'\Vert\le(4+\varepsilon')\Vert T\Vert,\quad\Vert S\Vert\le(1+\varepsilon')\Vert T\Vert\quad\text{and}\quad
\Vert S\circ R\Vert\le\varepsilon.
\end{gather*}
\end{lem}

\begin{proof}[Proof of Theorem~\ref{thm:main}]
We first observe that it is sufficient to prove the theorem in the case when $K$ is zero-dimensional. Namely, if $K$ is an arbitrary compact line then
$K\times\{0,1\}$, endowed with the lexicographic order, is a zero-dimensional compact line and we have an isometric embedding $\pi_1^*$
of $C(K)$ into $C\big(K\times\{0,1\}\big)$ induced by the first projection $\pi_1:K\times\{0,1\}\to K$.

In what follows, $K$ is a zero-dimensional compact line and $X$ is a Banach subspace of $C(K)$ having separable dual. Our goal is to show that the restriction
operator $\mathfrak r:\mathcal B\big(C(K),c_0\big)\to\mathcal B(X,c_0)$ is onto, where $\mathcal B(X_1,X_2)$ denotes the space of bounded operators
from $X_1$ to $X_2$. To this aim, we will prove that the closure of the image of
the unit ball of $\mathcal B\big(C(K),c_0\big)$ under $\mathfrak r$ contains the open ball of $\mathcal B(X,c_0)$ of radius $\frac18$. The surjectivity
of $\mathfrak r$ will then follow from \cite[Lemma~2.23]{Fabian2}.

Pick $L$ and $q:K\to L$ as in Lemma~\ref{thm:lemaCL} and denote by $R:X\to C(L)$ the restriction to $X$ of the isometry $(q^*)^{-1}:q^*C(L)\to C(L)$. Let $T_0\in\mathcal B(X,c_0)$ be fixed with $\Vert T_0\Vert<\frac18$ and let $\varepsilon>0$ be given. Since $C(L)$ is separable, by Sobczyk's Theorem,
there exists $T\in\mathcal B\big(C(L),c_0\big)$ with $\Vert T\Vert\le2\Vert T_0\Vert<\frac14$ and $T\circ R=T_0$. Apply Lemma~\ref{thm:hardoperators}
with $\varepsilon'>0$ chosen such that $(4+\varepsilon')\Vert T\Vert\le1$. Then $T'\in\mathcal B\big(C(K),c_0\big)$ satisfies
$\Vert T'\Vert\le1$ and $\Vert\mathfrak r(T')-T_0\Vert=\Vert S\circ R\Vert\le\varepsilon$.
\end{proof}

\begin{rem}\label{thm:rem8}
The proof of Theorem~\ref{thm:main} actually shows that given a compact line $K$, a Banach subspace $X$ of $C(K)$ with separable dual,
a bounded operator $T:X\to c_0$, and $\varepsilon>0$, there exists an extension $T':C(K)\to c_0$ of $T$ with $\Vert T'\Vert\le(8+\varepsilon)\Vert T\Vert$. The authors do not know if this estimate on the extension constant can be improved. We note that
in \cite[Theorem~3.1]{CT} we have shown that, under the additional assumption that $X$ is a Banach subalgebra of $C(K)$, the constant $8$ can be replaced with $2$.
\end{rem}

\end{section}

\begin{section}{Proof of main lemmas}\label{sec:technical}

The goal of this section is to prove Lemma~\ref{thm:hardoperators}. This will be achieved by translating the problem of decomposing the operator
$T$ as $T=T'\circ q^*+S$ into a problem of decomposing a sequence of measures $(\mu_n)_{n\ge1}$ as $\mu_n=\mu'_n+\nu_n$. First, we have to understand how to recognize, in terms of sequences of measures, which $c_0$-valued operators on $C(L)$ are of the form $T'\circ q^*$, for some $T':C(K)\to c_0$. This is the purpose of Lemma~\ref{thm:extensioncriterion}.

For the proof of Lemma~\ref{thm:extensioncriterion}, we need a more concrete representation of $M(K)$ when $K$ is a compact line. Given a compact line $K$ and a map $F:K\to\R$, the total variation $V(F)\in[0,+\infty]$ is defined exactly as in the case $K=[0,1]$. We denote by $\BV(K)$ the Banach space of functions $F:K\to\R$ of bounded variation
(i.e., functions $F$ with $V(F)<+\infty$) endowed with the norm:
\[\Vert F\Vert_{\BV}=\vert F(0)\vert+V(F).\]
Then:
\[\NBV(K)=\big\{F\in\BV(K):\text{$F$ is right-continuous}\big\}\]
is a closed subspace of $\BV(K)$. The following representation theorem is standard when $K=[0,1]$ and a proof in that case can be obtained, for instance,
by using Riemann--Stieltjes integration (\cite[4.32]{Taylor}). The proof in the case of a general compact line $K$ is similar. However, a little more care is
required in the definition of the Riemann--Stieltjes sum, because of the possible presence of consecutive points. We give the details below.
\begin{lem}\label{thm:MKNBVK}
Given a compact line $K$, the map:
\begin{equation}\label{eq:muFmu}
M(K)\ni\mu\longmapsto F_\mu\in\NBV(K)
\end{equation}
is a linear isometry, where $F_\mu$ is defined by $F_\mu(t)=\mu\big([0,t]\big)$, for all $t\in K$.
\end{lem}
\begin{proof}
Given $\mu\in M(K)$, we have that $\Vert F_\mu\Vert_{\BV}\le\Vert\mu\Vert$ and it follows from the regularity of $\mu$ that
$F_\mu$ is right-continuous. Using Riemann--Stieltjes integration, an inverse for \eqref{eq:muFmu} can be defined. More precisely,
for $f\in C(K)$ and $F\in\NBV(K)$, we define the Riemann--Stieltjes sum $\mathcal S(f,F;P)$ by setting:
\[\mathcal S(f,F;P)=f(0)F(0)+\sum_{i=0}^{n-1}f(t_{i+1})\big(F(t_{i+1})-F(t_i)\big),\]
where $P:0=t_0<t_1<\cdots<t_n=\max K$ is a partition of $K$. It is easy to prove that, for $f\in C(K)$ and $\varepsilon>0$, there
exists a partition $P$ of $K$ such that $f$ oscillates less than $\varepsilon$ on each interval $[t_i,t_{i+1}]$ of $P$ with
$\left]t_i,t_{i+1}\right[\ne\emptyset$. Using this fact, it follows that the Riemann--Stieltjes
integral is well-defined by $\int_Kf\,\dd F=\lim_P\mathcal S(f,F;P)$, with partitions $P$ of $K$ ordered by inclusion.
For $F\in\NBV(K)$, the bounded linear functional $f\mapsto\int_Kf\,\dd F$ on $C(K)$ is represented by a measure $\mu_F\in M(K)$ with
$\Vert\mu_F\Vert\le\Vert F\Vert_{\BV}$. By standard arguments, one shows that the operator $\NBV(K)\ni F\mapsto\mu_F\in M(K)$ is the inverse of \eqref{eq:muFmu}.
\end{proof}

\begin{lem}\label{thm:extensioncriterion}
Let $K$ and $L$ be compact lines and $q:K\to L$ be a continuous increasing surjection. Assume that $K$ is zero-dimensional. Set:
\[Q=\big\{t\in L:\vert q^{-1}(t)\vert>1\big\},\]
where $\vert\cdot\vert$ denotes the cardinality of a set.
Let $T:C(L)\to c_0$ be a bounded operator associated with a weak*-null sequence $(\mu_n)_{n\ge1}$ in $M(L)$. The following conditions are equivalent:
\begin{itemize}
\item[(a)] there exists a bounded operator $T':C(K)\to c_0$ with $T'\circ q^*=T$;
\item[(b)] there exists a countable subset $E$ of $Q$ such that $\mu_n\big([0,t]\big)\longrightarrow0$, for all $t\in Q\setminus E$;
\item[(c)] there exists a bounded operator $T':C(K)\to c_0$ with $T'\circ q^*=T$ and $\Vert T'\Vert\le2\Vert T\Vert$.
\end{itemize}
\end{lem}
\begin{proof}
Assume (a) and let us prove (b). Let $T'$ be associated with a weak*-null sequence $(\mu'_n)_{n\ge1}$ in $M(K)$.
Let $F_n\in\NBV(L)$ and $F'_n\in\NBV(K)$ correspond, respectively, to $\mu_n$ and $\mu'_n$ as in Lemma~\ref{thm:MKNBVK}.
From $T'\circ q^*=T$ we obtain $q_*(\mu'_n)=\mu_n$, for all $n\ge1$, where $q_*:M(K)\to M(L)$ denotes the adjoint of $q^*$. The operator $q_*$ is given by:
\[q_*(\mu')(B)=\mu'\big(q^{-1}[B]\big),\quad\mu'\in M(K),\]
for every Borel subset $B$ of $L$. It follows that:
\[F_n(t)=F'_n(b_t),\]
for all $t\in L$, where $b_t=\max\big(q^{-1}(t)\big)$. Since $K$ is zero-dimensional, for each $t\in Q$, there exists a right-isolated point
$a_t$ of $K$ in $q^{-1}(t)$. For all $t\in Q$, we have $F'_n(b_t)-F'_n(a_t)=\mu'_n\big(\left]a_t,b_t\right]\!\big)$ and, since the intervals
$\left]a_t,b_t\right]$ are disjoint, it follows that $F'_n(b_t)-F'_n(a_t)=0$, for all $t\in Q$ outside a countable subset $E_n$ of $Q$.
Set $E=\bigcup_{n=1}^\infty E_n$. For $t\in Q\setminus E$, we have $F_n(t)=F'_n(a_t)\longrightarrow0$, since $[0,a_t]$ is a clopen subset
of $K$ and $(\mu'_n)_{n\ge1}$ is weak*-null.

Now assume (b) and let us prove (c). Define $F_n$ and $b_t$ as above. For each $n\ge1$, set $G_n=F_n\circ q$. Then $G_n\in\NBV(K)$ and
$\Vert G_n\Vert_{\BV}=\Vert F_n\Vert_{\BV}$, so that $G_n$ corresponds to a measure $\nu_n\in M(K)$ with $\Vert\nu_n\Vert=\Vert\mu_n\Vert$
and the operator $S:C(K)\to\ell_\infty$ associated with $(\nu_n)_{n\ge1}$ satisfies $\Vert S\Vert=\Vert T\Vert$.
Moreover, since $F_n(t)=G_n(b_t)$, for all $t\in L$, we have $S\circ q^*=T$.
We will show that the quotient $C(K)/S^{-1}[c_0]$ is separable and it will follow from \cite[Proposition~2.2,~(a)]{CT} that $S\vert_{S^{-1}[c_0]}$ admits an extension $T':C(K)\to c_0$ with $\Vert T'\Vert\le2\Vert S\Vert$. Since $K$ is zero-dimensional, the set:
\begin{equation}\label{eq:chi0s}
\big\{\chilow{[0,s]}:\text{$s$ right-isolated in $K$}\big\}
\end{equation}
spans a dense subspace of $C(K)$. Let us prove that the image of \eqref{eq:chi0s} under the quotient map $C(K)\to C(K)/S^{-1}[c_0]$ is countable.
We claim that if $s\in K$ is right-isolated and $s\not\in q^{-1}[E]$ then $\chilow{[0,s]}\in S^{-1}[c_0]$. To prove the claim, note first that $S(\chilow{[0,s]})=\big(G_n(s)\big)_{n\ge1}$. In case $q(s)\in Q$, we have $q(s)\in Q\setminus E$ and therefore:
\begin{equation}\label{eq:GnFnqs}
G_n(s)=F_n\big(q(s)\big)\longrightarrow0.
\end{equation}
In case $q(s)\not\in Q$, since $s$ is right-isolated in $K$, we have that $q(s)$ is right-isolated in $L$ and then \eqref{eq:GnFnqs} follows from the fact that $(\mu_n)_{n\ge1}$
is weak*-null. This proves the claim. To conclude the proof of (c), simply observe that for $t\in E$ and $s_1,s_2\in q^{-1}(t)$ right-isolated in $K$,
we have:
\[\chilow{[0,s_1]}-\chilow{[0,s_2]}\in\Ker(S)\subset S^{-1}{[c_0]}.\qedhere\]
\end{proof}

\begin{rem}\label{thm:gambiarra}
In the proof of (a)$\Rightarrow$(b) in Lemma~\ref{thm:extensioncriterion}, the assumption that $K$ is zero-dimensional is only used to establish
that $q^{-1}(t)$ contains a right-isolated point of $K$, for all $t\in Q$.
\end{rem}

We are now ready to state the measure-theoretic version of Lemma~\ref{thm:hardoperators}.
\begin{lem}\label{thm:hardmeasures}
Let $L$ be a metrizable zero-dimensional compact line, $X$ be a Banach space with separable dual, $R:X\to C(L)$ be a bounded operator and
$(\mu_n)_{n\ge1}$ be a weak*-null sequence in $M(L)$ with $\sup_{n\ge1}\Vert\mu_n\Vert\le1$. Given $\varepsilon,\varepsilon'>0$, there exist
weak*-null sequences $(\mu'_n)_{n\ge1}$ and $(\nu_n)_{n\ge1}$ in $M(L)$ such that:
\begin{itemize}
\item[(a)] $\mu_n=\mu'_n+\nu_n$, for all $n\ge1$;
\item[(b)] $\Vert\nu_n\Vert\le1+\varepsilon'$ and $\Vert\mu'_n\Vert\le2+\varepsilon'$, for all $n\ge1$;
\item[(c)] $\Vert R^*(\nu_n)\Vert\le\varepsilon$, for all $n\ge1$;
\item[(d)] $\mu'_n\big([0,t]\big)\longrightarrow0$, for all $t$ outside a countable set.
\end{itemize}
\end{lem}

We now prove Lemma~\ref{thm:hardoperators} using Lemma~\ref{thm:hardmeasures}.
\begin{proof}[Proof of Lemma~\ref{thm:hardoperators}]
Without loss of generality, assume $\Vert T\Vert=1$. Denote by $(\mu_n)_{n\ge1}$ the weak*-null sequence in $M(L)$ associated with $T$.
Using Lemma~\ref{thm:hardmeasures} we obtain weak*-null sequences $(\mu'_n)_{n\ge1}$ and $(\nu_n)_{n\ge1}$ in $M(L)$ satisfying (a)---(d). Let $T'_0:C(L)\to c_0$ and $S:C(L)\to c_0$ be the bounded operators associated with $(\mu'_n)_{n\ge1}$ and $(\nu_n)_{n\ge1}$, respectively. By item~(a),
$T=T'_0+S$. Item~(b) implies:
\[\Vert T'_0\Vert\le2+\varepsilon',\quad\Vert S\Vert\le1+\varepsilon',\]
and item~(c) gives $\Vert S\circ R\Vert\le\varepsilon$.
Finally, Lemma~\ref{thm:extensioncriterion} and (d) yield a bounded operator $T':C(K)\to c_0$ such that $\Vert T'\Vert\le2\Vert T'_0\Vert$
and $T'\circ q^*=T'_0$.
\end{proof}

The proof of Lemma~\ref{thm:hardmeasures} is hard and it requires a few more lemmas. Let us first introduce some notation.

\medskip

Given a Banach space $X$, a compact Hausdorff space $L$ and a bounded operator $R:X\to C(L)$, we denote by $\phi^R:L\to X^*$
the map defined by:
\[\phi^R(p)=R^*(\delta_p),\quad p\in L.\]
Obviously, the map $\phi^R$ is continuous when $X^*$ is endowed with the weak*-topology and:
\[\Vert R\Vert=\sup_{p\in L}\Vert\phi^R(p)\Vert.\]

Recall that a Banach space $X$ is said to be {\em weak*-fragmentable\/} if for every nonempty bounded subset $B$ of $X^*$ and every $\varepsilon>0$, there exists a nonempty set $U$ weak*-open relatively to $B$ with $\diam(U)<\varepsilon$, where $\diam(U)$ denotes the diameter of $U$. It is well-known that $X$ is weak*-fragmentable if and only if $X$ is Asplund; in particular, if $X$ has separable dual then $X$ is weak*-fragmentable (see \cite[Theorem~11.8]{Fabian}).
\begin{lem}\label{thm:Halphas}
Let $X$ be a Banach space, $L$ be a compact Hausdorff space and $R:X\to C(L)$ be a bounded operator. Given $\delta>0$, define
by recursion on the ordinal $\alpha$ a decreasing family of closed subsets $H_\alpha$ of $L$ by setting $H_0=L$,
\[H_{\alpha+1}=\big\{p\in H_\alpha:\text{$\diam\big(\phi^R[V]\big)\ge\delta$, for every nhood $V$ of $p$ in $H_\alpha$}\big\},\]
and $H_\alpha=\bigcap_{\beta<\alpha}H_\beta$ if $\alpha$ is a limit ordinal. If $X$ is weak*-fragmentable, then there exists $\alpha$ such that $H_\alpha=\emptyset$.
\end{lem}
\begin{proof}
The conclusion will follow from the fact that, given an ordinal $\alpha$ with $H_\alpha\ne\emptyset$, the set $H_{\alpha+1}$ is properly contained
in $H_\alpha$. To prove the latter, note that $\phi^R[H_\alpha]$ is a nonempty bounded subset of $X^*$ and thus there exists a nonempty set $U$ weak*-open
relatively to $\phi^R[H_\alpha]$ with $\diam(U)<\delta$. Then $(\phi^R\vert_{H_\alpha})^{-1}[U]$ is a nonempty set, open relatively to $H_\alpha$,
and it is disjoint from $H_{\alpha+1}$.
\end{proof}

\begin{rem}
We note that the hierarchy of closed subsets $H_\alpha$ of $L$ defined in the statement of Lemma~\ref{thm:Halphas} coincides precisely with the oscillation hierarchy $P^\alpha_{\delta,\phi^R}$ of the map $\phi^R:L\to X^*$ defined in \cite[1.II]{KL} and that the least ordinal $\alpha$ with $H_\alpha=\emptyset$
is the $\delta$-oscillation rank $\beta(\phi^R,\delta)$ of the map $\phi^R$ defined in that article. Moreover, it is easily checked that
such $\alpha$ is bounded by the Szlenk index of the Banach space $X$. See \cite{Szlenk} for Szlenk's original definition of the index,
\cite{LancienSurvey} for a survey, and \cite[Proposition~3.3]{Lancien} for the equivalence (in the case of spaces not containing $\ell_1$)
between Szlenk's original definition and the definition appearing in \cite{LancienSurvey}.
We observe that the Szlenk index of a Banach space is well-defined if and only if
the space is Asplund (\cite[Theorem~2]{LancienSurvey}) and it is countable if the space has separable dual (\cite[Theorem~1]{LancienSurvey}).
\end{rem}

\begin{lem}\label{thm:flower}
Let $X$ be a Banach space, $L$ be a compact Hausdorff space and $R:X\to C(L)$ be a bounded operator. If $\mu\in M(L)$ satisfies $\mu(L)=0$ then:
\[\Vert R^*(\mu)\Vert\le\frac12\diam\big(\phi^R[\supp\mu]\big)\Vert\mu\Vert.\]
\end{lem}
\begin{proof}
We can assume without loss of generality that $\supp\mu=L$. Namely, setting $H=\supp\mu$, $\mu'=\mu\vert_H$, and $R'=\rho_H\circ R$, we have that $\phi^{R'}=\phi^R\vert_H$, $(R')^*(\mu')=R^*(\mu)$, and $\Vert\mu\Vert=\Vert\mu'\Vert$, where $\rho_H:C(L)\to C(H)$ denotes the restriction operator.

Let $\mu=\mu^+-\mu^-$ be the Jordan decomposition of the measure $\mu$. Then $\Vert\mu\Vert=\mu^+(L)+\mu^-(L)$ and $\mu^+(L)=\mu^-(L)$, so that
$\mu^+$ and $\mu^-$ are regular nonnegative measures of norm $\frac12\Vert\mu\Vert$. We can obviously assume that $\Vert\mu\Vert=2$, which implies that $\mu^+$ and $\mu^-$ are regular probability measures and therefore they belong to the weak*-closed convex hull of $\big\{\delta_p:p\in L\big\}$. Since $R^*$
is weak*-continuous and linear, $R^*(\mu^+)$ and $R^*(\mu^-)$ belong to the weak*-closed convex hull of $\phi^R[L]$. Noting that
the diameter of a subset of $X^*$ is equal to the diameter of its weak*-closed convex hull, we obtain:
\[\Vert R^*(\mu)\Vert=\Vert R^*(\mu^+)-R^*(\mu^-)\Vert\le\diam\big(\phi^R[L]\big),\]
concluding the proof.
\end{proof}

\begin{lem}\label{thm:skeleton}
Let $I$ be a compact line and $H$ be a nonempty closed subset of $I$. Given $\mu\in M(I)$, there exists $\nu\in M(I)$ such that:
\begin{itemize}
\item[(A)] $\supp\nu\subset H$;
\item[(B)] $\nu(I)=\mu(I)$;
\item[(C)] $\Vert\nu\Vert\le\Vert\mu\Vert$;
\item[(D)] $\nu\big([\min I,t]\big)=\mu\big([\min I,t]\big)$, for all $t\in H\setminus\{\max H\}$.
\end{itemize}
The assumption $H\ne\emptyset$ can be dropped if $\mu(I)=0$.
\end{lem}
\begin{proof}
Denote by $S$ the set of points of $H$ that are left-isolated in $H$. For each $t\in S$, set $I_t=\left[0,t\right[$, if $t=\min H$,
and if $t\ne\min H$, set $I_t=\left]t^-,t\right[$, where $t^-\in H$ denotes the predecessor of $t$ in $H$. Note that $I$ is equal to the disjoint union
of $H$, the intervals $I_t$, $t\in S$, and the interval $\left]\max H,\max I\right]$. The measure $\nu$ is defined by:
\[\nu=\mu_H+\sum_{t\in S}\mu(I_t)\delta_t+\mu\big(\left]\max H,\max I\right]\!\big)\delta_{\max H},\]
where $\mu_H\in M(I)$ is given by $\mu_H(B)=\mu(H\cap B)$, for every Borel subset $B$ of $I$. Conditions (A)---(D) are easily checked keeping in mind that,
since $\mu$ is regular and each $I_t$ is open, there exists a countable subset $S_0$ of $S$ such that $\vert\mu\vert\big(\bigcup_{t\in S\setminus S_0}I_t\big)=0$.
\end{proof}

To state the next lemma and prove Lemma~\ref{thm:hardmeasures}, it is convenient to introduce the following terminology.
By a {\em clopen-partition\/} of a compact line $L$ we mean a finite set $P$ of right-isolated points of $L$ such that $\max L$ is in $P$. If $P=\{b_1,\ldots,b_m\}$, with $b_1<\cdots<b_m=\max L$, we write $\overline P=\big\{[0,b_1],\left]b_1,b_2\right],\ldots,\left]b_{m-1},b_m\right]\!\big\}$,
so that $\overline P$ is a partition of $L$ into clopen intervals. For $t\in L$, we denote by $\overline P(t)$ the unique $I\in\overline P$ such that
$t\in I$.
\begin{lem}\label{thm:lematildemu}
Let $L$ be a compact line and $P$ be a clopen-partition of $L$. Given $\mu\in M(L)$, there exists $\tilde\mu\in M(L)$ satisfying the following conditions:
\begin{itemize}
\item[(i)] $\tilde\mu(I)=0$, for all $I\in\overline P$;
\item[(ii)] $\tilde\mu\big([0,t]\big)=\mu\big([0,t]\big)$, for all $t\in L\setminus P$;
\item[(iii)] $\Vert\tilde\mu\Vert\le\Vert\mu\Vert+2\sum_{b\in P}\big\vert\mu\big([0,b]\big)\big\vert$.
\end{itemize}
\end{lem}
\begin{proof}
Simply define $\tilde\mu$ by setting:
\[\tilde\mu=\mu-\sum_{b\in P}\mu\big([0,b]\big)\delta_b+\!\!\!\!\!\sum_{\substack{b\in P\\b\ne\max L}}\!\!\!\!\mu\big([0,b]\big)\delta_{b^+},\]
where $b^+$ denotes the successor of $b$ in $L$.
\end{proof}

\begin{proof}[Proof of Lemma~\ref{thm:hardmeasures}]
Consider the decreasing family of closed subsets $H_\alpha$ of $L$ defined as in Lemma~\ref{thm:Halphas}, with $\delta=\frac{2\varepsilon}{1+\varepsilon'}$.
For each clopen interval $I\subset L$, set:
\begin{equation}\label{eq:defalphaI}
\alpha(I)=\min\big\{\alpha:\diam\big(\phi^R[H_\alpha\cap I]\big)<\delta\big\}.
\end{equation}
The ordinal $\alpha(I)$ is well-defined since there exists $\alpha$ with $H_\alpha=\emptyset$ and, for such $\alpha$, we have $\diam\big(\phi^R[H_\alpha\cap I]\big)<\delta$.

Since $L$ is second countable, it has only a countable number of right-isolated points. Let $(P_k)_{k\ge1}$ be an increasing sequence of clopen-partitions of $L$ such that $\bigcup_{k=1}^\infty P_k$ is equal to the set of all right-isolated points of $L$. It follows from the fact that $L$ is zero-dimensional
that, for each $t\in L$, the set $\big\{\overline{P_k}(t):k\ge1\big\}$ is a fundamental system of neighborhoods of $t$. Since $(\mu_n)_{n\ge1}$ is weak*-null,
for each $k\ge1$ there exists $n_k\ge1$ such that:
\begin{equation}\label{eq:epsilonlinha}
2\sum_{b\in P_k}\big\vert\mu_n\big([0,b]\big)\big\vert\le\varepsilon',
\end{equation}
for all $n\ge n_k$. The sequence $(n_k)_{k\ge1}$ can be chosen to be strictly increasing.

Now let $n\ge1$ be given and let us define $\nu_n$ and $\mu'_n$. If $n<n_1$, set $\nu_n=0$ and $\mu'_n=\mu_n$, so that (a)---(c) hold.
If $n\ge n_1$, pick the only $k\ge1$ such that $n_k\le n<n_{k+1}$. Using Lemma~\ref{thm:lematildemu} with $\mu=\mu_n$ and $P=P_k$, we get $\tilde\mu_n\in M(L)$ satisfying (i)---(iii). For each $I\in\overline{P_k}$, we have $\tilde\mu_n(I)=0$; apply Lemma~\ref{thm:skeleton} to the compact line $I$, the closed subset $H_{\alpha(I)}\cap I$ of $I$, and the measure $\tilde\mu_n\vert_I$ obtaining $\nu_n^I\in M(I)$ satisfying (A)---(D). Define $\nu_n\in M(L)$ so that $\nu_n\vert_I=\nu_n^I$, for all $I\in\overline{P_k}$, and set $\mu'_n=\mu_n-\nu_n$. Then (a) holds. The following simple computation proves (b):
\begin{align*}
\Vert\nu_n\Vert&=\sum_{I\in\overline{P_k}}\Vert\nu_n^I\Vert\stackrel{\text{(C)}}\le\sum_{I\in\overline{P_k}}\Vert\tilde\mu_n\vert_I\Vert
=\Vert\tilde\mu_n\Vert\\[4pt]
&\stackrel{\text{(iii)}}\le\Vert\mu_n\Vert+2\sum_{b\in P_k}\big\vert\mu_n\big([0,b]\big)\big\vert\stackrel{\eqref{eq:epsilonlinha}}\le1+\varepsilon'.
\end{align*}
This yields also $\Vert\mu'_n\Vert\le2+\varepsilon'$. To prove (c), note first that:
\[\nu_n=\sum_{I\in\overline{P_k}}\nu_n^I,\]
where $\nu_n^I\in M(I)$ is identified with its extension to $L$ that vanishes identically outside $I$.
By (A), we have $\supp\nu_n^I\subset H_{\alpha(I)}\cap I$; moreover:
\begin{equation}\label{eq:nunIzero}
\nu_n^I(I)\stackrel{\text{(B)}}=\tilde\mu_n(I)\stackrel{\text{(i)}}=0.
\end{equation}
Then, using Lemma~\ref{thm:flower}, we obtain:
\begin{align*}
\Vert R^*(\nu_n)\Vert\le\sum_{I\in\overline{P_k}}\Vert R^*(\nu_n^I)\Vert&\le\frac12\sum_{I\in\overline{P_k}}\diam\big(\phi^R[H_{\alpha(I)}\cap I]\big)\Vert\nu_n^I\Vert\\
&\stackrel{\eqref{eq:defalphaI}}\le\frac12\,\delta\,\Vert\nu_n\Vert\le\varepsilon.
\end{align*}

Now let us prove that $(\nu_n)_{n\ge1}$ (and hence $(\mu'_n)_{n\ge1}$) is weak*-null. Since $L$ is zero-dimensional and $(\nu_n)_{n\ge1}$ is bounded,
it suffices to fix a right-isolated $t$ in $L$ and check that $\nu_n\big([0,t]\big)\longrightarrow0$. Take $k_0\ge1$ with $t\in P_{k_0}$ and let us show
that $\nu_n\big([0,t]\big)=0$ for $n\ge n_{k_0}$. If $n\ge n_{k_0}$, we have $n_k\le n<n_{k+1}$ for some $k\ge k_0$. Then $t\in P_k$ and $[0,t]$
is a disjoint union of elements of $\overline{P_k}$. Using \eqref{eq:nunIzero}, we obtain $\nu_n\big([0,t]\big)=0$.

Finally, let us prove (d). Set:
\[E=\big\{\!\max\!\big(H_{\alpha(I)}\cap I\big):\text{$I\in\overline{P_k}$, $k\ge1$ and $H_{\alpha(I)}\cap I\ne\emptyset$}\big\},\]
so that $E$ is a countable subset of $L$. Let $t\in L\setminus E$ be fixed. We claim that $\mu'_n\big([0,t]\big)\longrightarrow0$.
Let $\beta$ be the largest ordinal with $t\in H_\beta$. Then $t\not\in H_{\beta+1}$ and thus there exists a neighborhood $V$
of $t$ in $H_\beta$ with $\diam\big(\phi^R[V]\big)<\delta$. Since $\big\{\overline{P_k}(t):k\ge1\big\}$ is a fundamental system
of neighborhoods of $t$ in $L$, there exists $k_0\ge1$ with $H_\beta\cap\overline{P_{k_0}}(t)\subset V$. The claim will be proven once
we have established that $\mu'_n\big([0,t]\big)=0$ for $n\ge n_{k_0}$. Given $n\ge n_{k_0}$, there exists $k\ge k_0$ with
$n_k\le n<n_{k+1}$. Setting $I_t=\overline{P_k}(t)$, we have $I_t\subset\overline{P_{k_0}}(t)$, so that $H_\beta\cap I_t\subset V$
and $\diam\big(\phi^R[H_\beta\cap I_t]\big)<\delta$. It follows that $\alpha(I_t)\le\beta$ and thus $t\in H_{\alpha(I_t)}\cap I_t$.
Since $t\not\in E$, we have that $t\ne\max\!\big(H_{\alpha(I_t)}\cap I_t\big)$; in particular $t\ne\max I_t$ and thus $t\not\in P_k$.
We compute:
\[\nu_n\big([0,t]\big)\stackrel{\eqref{eq:nunIzero}}=\nu_n^{I_t}\big([\min I_t,t]\big)\stackrel{\text{(D)}}=
\tilde\mu_n\big([\min I_t,t]\big)\stackrel{\text{(i)}}=\tilde\mu_n\big([0,t]\big)\stackrel{\text{(ii)}}=\mu_n\big([0,t]\big).\]
Hence $\mu'_n\big([0,t]\big)=0$. This proves the claim and concludes the proof.
\end{proof}

\end{section}


\begin{thebibliography}{99}





\bibitem{CT} C. Correa \& D. V. Tausk, On extensions of $c_0$-valued operators, {\em J. Math.\ Anal.\ Appl.\/} 405, 2013, pgs.\ 400---408.


\bibitem{Engelking} R. Engelking, {\em General Topology}, Sigma Series in Pure Math.\ 6, Heldermann Verlag, Berlin, 1989.

\bibitem{Fabian2} M. Fabian, P. Habala, P. H\'ajek, V. Montesinos, J. Pelant, and V. Zizler, {\em Functional Analysis and Infinite-Dimensional
Geometry}, CMS Books in Math., Springer, New York, 2001.

\bibitem{Fabian} M. Fabian, P. Habala, P. H\'ajek, V. Montesinos, and V. Zizler, {\em Banach Space Theory: The Basis for Linear and Nonlinear
Analysis}, CMS Books in Math., Springer, New York, 2011.

\bibitem{Godefroy} G. Godefroy, {\em review\/} of the paper \cite{Patterson}, MR1245820.

\bibitem{KK} O. Kalenda \& W. Kubi\'s, Complementation in spaces of continuous functions on compact lines, {\em J. Math.\ Anal.\ Appl.\/} 386, 2012,
pgs.\ 241---257.

\bibitem{KK2} O. Kalenda \& W. Kubi\'s, The structure of Valdivia compact lines, {\em Topology Appl.\/} 157, 2010, pgs.\ 1142---1151.

\bibitem{KL} A. S. Kechris \& A. Louveau, A classification of Baire class $1$ functions, {\em Trans.\ Amer.\ Math.\ Soc.\/} 318 (1), 1990, pgs.\ 209---236.

\bibitem{Kozsmider} P. Koszmider, Some topological invariants and biorthogonal systems in Banach spaces, {\em Extracta Math.\/} 26 (2), 2011, pgs.\ 271---294.

\bibitem{Kubis} W. Kubi\'s, Linearly ordered compacta and Banach spaces with a projectional resolution of the identity,
{\em Topology Appl.\/} 154 (3), 2007, pgs.\ 749---757.

\bibitem{LancienSurvey} G. Lancien, A survey on the Szlenk index and some of its applications, {\em Rev.\ R. Acad.\ Cienc.\ Exactas Fis.\ Nat.\ Ser.\ A Mat.\ RACSAM\/} 100 (1--2), 2006, pgs.\ 209---235.

\bibitem{Lancien} G. Lancien, Dentability indices and locally uniformly convex renormings, {\em Rocky Mt.\ J. Math.\/} 23 (2), 1993, pgs.\ 635---647.

\bibitem{Molto} A. Molt\'o, On a theorem of Sobczyk, {\em Bull.\ Aust.\ Math.\ Soc.\/} 43 (1), 1991, pgs.\ 123---130.

\bibitem{Patterson} W. M. Patterson, Complemented $c_0$-subspaces of a non-separable $C(K)$-space,
{\em Canad.\ Math.\ Bull.\/} 36 (3), 1993, pgs.\ 351---357.

\bibitem{Rudin} W. Rudin, {\em Real and complex analysis}, McGraw-Hill series in higher math., McGraw-Hill, 1987.





\bibitem{Sobczyk} A. Sobczyk, Projection of the space $(m)$ on its subspace $(c_0)$, {\em Bull.\ Amer.\
Math.\ Soc.\/} 47, 1941, pgs.\ 938---947.

\bibitem{Szlenk} W. Szlenk, The nonexistence of a separable reflexive Banach space universal for all separable reflexive Banach spaces, {\em Studia
Math.\/} 30, 1968, pgs.\ 53---61.

\bibitem{Taylor} A. E. Taylor, {\em Introduction to Functional Analysis}, John Wiley, 1958.

\bibitem{Valdivia} M. Valdivia, Projective resolution of identity in $C(K)$ spaces, {\em Arch.\ Math.\/} 54, 1990, pgs.\ 493---498.


\bibitem{Veech} W. A. Veech, Short proof of Sobczyk's Theorem, {\em Proc.\ Amer.\ Math.\ Soc.\/} 28, 1971, pgs.\ 627---628.



\end{thebibliography}
\end{document}